\documentclass[12pt]{amsart}

\usepackage[utf8]{inputenc}
\usepackage[english]{babel}

\usepackage[letterpaper,top=2cm,bottom=2cm,left=3cm,right=3cm,marginparwidth=1.75cm]{geometry}

\usepackage{amsmath}
\usepackage{graphicx}
\usepackage[colorlinks=true, allcolors=blue]{hyperref}

\usepackage{amssymb} % Some useful typeface
\usepackage{comment}
\usepackage{indentfirst}
\usepackage{csquotes}
\usepackage{theoremref}
\usepackage{hyperref}
\usepackage{geometry}

\hypersetup{hidelinks, colorlinks=true, allcolors=cyan, pdfstartview=Fit, breaklinks=true}
\usepackage{amsthm}

\theoremstyle{plain}
\newtheorem{theorem}{Theorem}
\numberwithin{theorem}{section}
\newtheorem{corollary}[theorem]{Corollary}
\newtheorem*{corollary*}{Corollary}
\newtheorem*{Example*}{Example}

\newtheorem{proposition}[theorem]{Proposition}

\newtheorem{conjecture}[theorem]{Conjecture}
\theoremstyle{definition}

\newtheorem*{def*}{Definition}
\newtheorem*{theorem*}{Theorem}

\newtheorem*{definition*}{Definition}

\theoremstyle{remark}
\newtheorem*{remark}{Remark}

\newcommand{\bracket}[1]{\left( #1 \right)}
\newcommand{\modulo}[3]{#1\equiv#2\ \bracket{\mathrm{mod}\ #3}}

\newcommand{\bbZ}[0]{\mathbb Z}

\numberwithin{equation}{section}

\title{\textbf{Distribution of 3-regular and 5-regular partitions}}
\author{QI-YANG ZHENG}
\date{} % Don't show current time

\address{Department of Mathematics, Sun Yat-sen University(Zhuhai Campus), Zhuhai}
\email{zhengqy29@mail2.sysu.edu.cn}

\begin{document}
\maketitle

\begin{abstract}
In this paper we study the function $b_3(n)$ and $b_5(n)$, which denote the number of $3$-regular partitions and $5$-regular partitions of $n$ respectively. Using the theory of modular forms, we prove several arithmetic properties of $b_3(n)$ and $b_5(n)$ modulo primes greater than $3$.
\end{abstract}

% \tableofcontents

~

\section{Introduction}

The number of partitions of $n$ in which no parts are multiples of $k$ is denoted by $b_k(n)$, the $k$-regular partitions. $b_k(n)$ is also the number of partitions of $n$ into at most $k-1$ copies of each part.

We agree that $b_3(0)=b_5(0)=1$ for convenience. Moreover, let $b_3(n)=b_5(n)=0$ if $n\not\in\bbZ_{\geq0}$. The $k$-regular partitions has generating function as follows:

\begin{equation}
    \notag
    \sum_{n=0}^\infty b_k(n)q^n=\prod_{n=1}^\infty\frac{1-q^{kn}}{1-q^n}.
\end{equation}

~

In 1919, Ramanujan found three remarkable congruences of $p(n)$ as follows
\begin{equation}
\notag
\begin{aligned}
p(5n+4)&\equiv0\ (\mathrm{mod}\ 5),\\
p(7n+5)&\equiv0\ (\mathrm{mod}\ 7),\\
p(11n+6)&\equiv0\ (\mathrm{mod}\ 11).
\end{aligned}
\end{equation}

In 2000, Ono \cite{ono2000distribution} proved that for each prime number $m\geq5$, there exists infinitely many arithmetic sequences $An+B$ such that
$$p(An+B)\equiv0\ (\mathrm{mod}\ m).$$

We call such congruences Ramanujan-type congruences. Subsequently Lovejoy \cite{lovejoy2001divisibility} gave similar results for the function $Q(n)$, the number of partitions of $n$ into distinct parts. Following strategies of Ono and Lovejoy, we prove the following theorem.

~

\begin{theorem}
\label{infinitely many Ramanujan-type congruences}
For each prime $m\geq5$, there are infinitely many Ramanujan-type congruences of $b_3(n)$ and $b_5(n)$ modulo $m$.
\end{theorem}

~

Lovejoy and Penniston \cite{lovejoy20013} study the distribution of $b_3(n)$ modulo $3$. Recently, Keith and Zanello \cite{keith2022parity} study the parity of $b_3(n)$. As for the $5$-regular partitions, Calkin et al. \cite{calkin2008divisibility}, Hirschhorn and Sellers \cite{hirschhorn2010elementary} study the parity of $b_5(n)$. Gordon and Ono \cite{gordon1997divisibility} study the distribution of $b_5(n)$ modulo $5$. Moreover, they prove that
\begin{equation}
    \label{b_5(5n+4)}
    b_5(5n+4)\equiv0\ (\mathrm{mod}\ 5).
\end{equation}

Up to now, we only know the distribution of $b_3(n)$ and $b_5(n)$ modulo primes mentioned above. In this paper, we study the distribution of $b_3(n)$ and $b_5(n)$ modulo primes $m\geq5$. It is noteworthy that we still do not know anything about $b_5(n)$ modulo $3$.

As a natural corollary of Theorem \ref{infinitely many Ramanujan-type congruences}, we have

~

\begin{corollary}
If $m\geq5$ is a prime and $k=3,5$, then there are infinitely many positive integers $n$ for which
$$b_k(n)\equiv0\ (\mathrm{mod}\ m).$$

\noindent
More precisely, we have
$$\#\{ 0\leq n\leq X\ :\ b_k(n)\equiv0\ (\mathrm{mod}\ m) \}\gg X.$$
\end{corollary}

~

For other residue classes $i\not\equiv0\ (\mathrm{mod}\ m)$, we provide a useful criterion to verify whether there are infinitely many $n$ such that $b_k(n)\equiv i\ (\mathrm{mod}\ m)$.

~

\begin{proposition}
\label{other residue classes}
If $m\geq5$ is a prime and there is one $k\in\mathbb{Z}$ such that
$$b_3\left( mk+\frac{m^2-1}{12} \right)\equiv e\not\equiv0\ (\mathrm{mod}\ m),$$

\noindent
then for each $i=1,2,\cdots,m-1$, we have
$$\#\{ 0\leq n\leq X\ :\ b_3(n)\equiv i\ (\mathrm{mod}\ m) \}\gg \frac{X}{\log X}.$$

\noindent
Moreover, if such $k$ exists, then $k<18(m-1)$.
\end{proposition}

We obtain similar results for $b_5(n)$.

\begin{proposition}
\label{other residue classes2}
Let $m\geq5$ be a prime. If there exists one $k\in\mathbb{Z}$ such that
$$b_5\left( mk+\frac{m^2-1}{6} \right)\equiv e\not\equiv0\ (\mathrm{mod}\ m),$$

\noindent
then for each $i=1,2,\cdots,m-1$, we have
$$\#\{ 0\leq n\leq X\ |\ b_5(n)\equiv i\ (\mathrm{mod}\ m) \}\gg \frac{X}{\log X}.$$

\noindent
Moreover, if such $k$ exists, then $k<10(m-1)$.
\end{proposition}

~

\begin{remark}
The congruence \eqref{b_5(5n+4)} show that our criterion is inapplicable for the case $m=5$. However the case $m=5$ is studied in \cite{gordon1997divisibility}.
\end{remark}

~

\section{Preliminaries on modular forms}

\noindent
First we introduce the $U$ operator. If $j$ is a positive integer, then
\begin{equation}
    \notag
    \left( \sum_{n=0}^\infty a(n)q^n \right)\ |\ U(j):=\sum_{n=0}^\infty a(jn)q^n.
\end{equation}

\noindent
Recalling that Dedekind's eta function is defined by
\begin{equation}
    \notag
    \eta(z)=q^\frac{1}{24}\prod_{n=1}^\infty (1-q^n),
\end{equation}

\noindent
where $q=e^{2\pi iz}$.

~

If $m$ is a prime, then let $M_{k}(\Gamma_0(N),\chi)_m$(resp. $S_{k}(\Gamma_0(N),\chi)_m$) denote the $\mathbb{F}_m$-vector space of the reductions mod $m$ of the $q$-expansions of modular forms(resp. cusp forms) in $M_{k}(\Gamma_0(N),\chi)$(resp. $S_{k}(\Gamma_0(N),\chi)$) with integer coefficients.

Sometimes, we will use the notation $a\equiv_m b$ in the place of $a\equiv b\ (\mathrm{mod}\ m)$ for convenience.

We need the following theorem to construct modular forms \cite[Theorem 3]{gordon1993multiplicative}:

\begin{theorem}[B. Gordon, K. Hughes]
\label{eta-quotient}
Let
$$f(z)=\prod_{\delta|N}\eta^{r_\delta}(\delta z)$$

\noindent
be a $\eta$-quotient provided

~

\noindent
$\mathrm{(\romannumeral1)}$ $$\sum_{\delta|N}\delta r_\delta\equiv0\ (\mathrm{mod}\ 24);$$
$\mathrm{(\romannumeral2)}$ $$\sum_{\delta|N}\frac{Nr_\delta}{\delta}\equiv0\ (\mathrm{mod}\ 24);$$
$\mathrm{(\romannumeral3)}$ $$k:=\frac{1}{2}\sum_{\delta|N}r_\delta\in\mathbb{Z},$$

\noindent
then

$$f\left(\frac{az+b}{cz+d}\right)=\chi(d)(cz+d)^kf(z),$$

\noindent
for each $\begin{pmatrix}
 a & b\\
 c & d
\end{pmatrix}\in\Gamma_0(N)$ and $\chi$ is a Dirichlet character $(\mathrm{mod}\ N)$ defined by
$$\chi(n):=\left( \frac{(-1)^k\prod_{\delta|N}\delta^{r_\delta}}{n} \right),\ if\ n>0\ and\ (n,6)=1.$$
\end{theorem}

~

If $f(z)$ is holomorphic (resp. vanishes) at all cusps of $\Gamma_0(N)$, then $f(z)\in M_k(\Gamma_0(N),\chi)$ (resp. $S_k(\Gamma_0(N),\chi)$), since $\eta(z)$ is never vanishes on $\mathcal{H}$. The following theorem (c.f. \cite{martin1996multiplicative}) provide a useful criterion for compute the orders of an $\eta$-quotient at all cusps of $\Gamma_0(N)$.

\begin{theorem}[Y. Martin]

\label{order of cusp}
Let $c$, $d$ and $N$ be positive integers with $d\,|\,N$ and $(c,d)=1$. If $f(z)$ is an $\eta$-quotient satisfying the conditions of Theorem \ref{eta-quotient}, then the order of vanishing of $f(z)$ at the cusp $c/d$ is
$$\frac{N}{24}\sum_{\delta|N}\frac{r_\delta(d^2,\delta^2)}{\delta(d^2,N)}.$$

\end{theorem}

\section{Ramanujan-type congruences}

\noindent
In this section, we will prove Theorem \ref{infinitely many Ramanujan-type congruences} via theory of modular forms. However, the generating function of regular partition function is not a modular form. But for primes $m\geq5$, it turns out that for a properly chosen function $h_m(n)$, then
\begin{equation}
    \notag
    \sum_{n=0}^\infty b_k(h_m(n))q^n
\end{equation}

\noindent
is the Fourier expansion of a cusp form modulo $m$. In fact, we have

~

\begin{theorem}
\label{cusp form1}
Let $m\geq5$ be a prime, then
\begin{equation}
    \notag
    \sum_{n=0}^\infty b_3\left( \frac{mn-1}{12} \right)q^n\in S_{3m-3}(\Gamma_0(432),\chi_{12})_m,
\end{equation}

\noindent
where $\chi_{12}(n)=\left( \frac{n}{3} \right)\left( \frac{-4}{n} \right)$.
\end{theorem}

~

\begin{theorem}
\label{cusp form2}
Let $m\geq5$ be a prime, then
\begin{equation}
    \notag
    \sum_{n=0}^\infty b_5\left( \frac{mn-1}{6} \right)q^n\in S_{2m-2}(\Gamma_0(180),\chi_5)_m,
\end{equation}

\noindent
where $\chi_{5}(n)=\left( \frac{n}{5} \right)$.
\end{theorem}

~

\begin{proof}[Proof of Theorem \ref{cusp form1}]

We begin with an $\eta$-quotient
\begin{equation}
    \notag
    f(m;z):=\frac{\eta(3z)}{\eta(z)}\eta^a(3mz)\eta^b(mz),
\end{equation}

\noindent
where $m':=(m\ \mathrm{mod}\ 12)$, $a:=9-m'$ and $b:=m'-3$.

It is easy to verify that $f(m;z)\equiv_m\eta^{am+1}(3z)\eta^{bm-1}(z)$ satisfies the conditions of Theorem \ref{eta-quotient}. Moreover, one can compute via Theorem \ref{order of cusp} that $\eta^{am+1}(3z)\eta^{bm-1}(z)$ has the minimal order of vanishing of $(m(3a+b)+2)/24$ at the cusp $\infty$ and $(m(a+3b)-2)/24$ at the cusp $0$.

~

Since $(m(3a+b)+2)/24=(m(12-m')+1)/12>0$ and $(m(a+3b)-2)/24=(mm'-1)/12>0$, $\eta^{am+1}(3z)\eta^{bm-1}(z)\in S_{3m}(\Gamma_0(3),\chi_3)$, where $\chi_3(n)=\left(\frac{n}{3}\right)$. On the other hand,
\begin{equation}
    \notag
    f(m;z)=\sum_{n=0}^\infty b_3(n)q^{n+\frac{m(3a+b)+2}{24}}\cdot\prod_{n=1}^\infty(1-q^{3mn})^a(1-q^{mn})^b.
\end{equation}

\noindent
Thus,
\begin{equation}
    \label{after U(m)}
    \begin{aligned}
    &\ \ \ \ \eta^{am+1}(3z)\eta^{bm-1}(z)\ |\ U(m)\\
    &\equiv_m\left(\sum_{n=0}^\infty b_3(n)q^{n+\frac{m(3a+b)+2}{24}}\ |\ U(m)\right)\cdot\prod_{n=1}^\infty(1-q^{3n})^a(1-q^{n})^b.
    \end{aligned}
\end{equation}

~

\noindent
As for LHS of (\ref{after U(m)}),
$$\sum_{n=0}^\infty b_3(n)q^{n+\frac{m(3a+b)+2}{24}}\ |\ U(m)={\sum_{n\geq0}}^* b_3(n)q^{\frac{24n+m(3a+b)+2}{24m}},$$

\noindent
where ${\sum}^*$ means take integral power coefficients of $q$, i.e.
$$24n+m(3a+b)+2\equiv0\ (\mathrm{mod}\ 24m).$$

\noindent
It is easy to check that $24\,|\,24n+m(3a+b)+2$. Thus the condition becomes $m\,|\,12n+1$.

As for RHS of (\ref{after U(m)}), we have
$$\eta^{am+1}(3z)\eta^{bm-1}(z)\ |\ U(m)\equiv_m\eta^{am+1}(3z)\eta^{bm-1}(z)\ |\ T(m),$$

\noindent
where $T(m)$ denotes usual Hecke operator acting on $S_{3m}(\Gamma_0(3),\chi_3)$.

Now we analyze the $\eta$-product $\eta^6(z)\eta^6(3z)$. By Theorem \ref{eta-quotient} and Theorem \ref{order of cusp}, $\eta^6(z)\eta^6(3z)$ is a cusp form of weight $6$ and level $3$ and has the minimal order of vanishing of $1$ at the two cusps of $\Gamma_0(3)$. Since $\eta(z)$ never vanishes on $\mathcal{H}$, we can write $\eta^{am+1}(3z)\eta^{bm-1}(z)\ |\ T(m)=\eta^6(z)\eta^6(3z)g(m;z)$, where $g(m;z)\in M_{3m-6}(\Gamma_0(3),\chi_3)$.

In summary, we have
\begin{equation}
    \label{b_3(n)}
    \sum_{\genfrac{}{}{0pt}{}{n\geq0}{m|12n+1}} b_3(n)q^{\frac{24n+m(3a+b)+2}{24m}}\equiv_m\frac{\eta^6(z)\eta^6(3z)g(m;z)}{\prod_{n=1}^\infty(1-q^{3n})^a(1-q^{n})^b}.
\end{equation}

Replacing $q$ by $q^{12}$ and then multiplying by $q^{-(3a+b)/2}$ on both sides of (\ref{b_3(n)}), obtaining
$$\sum_{\genfrac{}{}{0pt}{}{n\geq0}{m|12n+1}} b_3(n)q^{\frac{12n+1}{m}}\equiv_m\eta^{6-a}(36z)\eta^{6-b}(12z)g(m;12z),$$

\noindent
namely,
$$\sum_{n=0}^\infty b_3\left(\frac{mn-1}{12}\right)q^n\equiv_m\eta^{6-a}(36z)\eta^{6-b}(12z)g(m;12z).$$

Using Theorem \ref{eta-quotient} and Theorem \ref{order of cusp} again, one can verity that $\eta^{6-a}(36z)\eta^{6-b}(12z)$ is a cusp form of weight $3$ and level $432$ and has the minimal order of vanishing of $m'$ at the cusps $c/d$ if $d=1,2,3,4,6,8,12,16,24,48$ and $12-m'$ if $d$ is other divisor of $432$.

Therefore we obtain
$$\eta^{6-a}(36z)\eta^{6-b}(12z)\in S_3\left(\Gamma_0(432),\chi_4\right),$$

\noindent
where $\chi_4(n)=\left( \frac{-3}{n} \right)$. Together with $g(m;12z)\in M_{3m-6}(\Gamma_0(36),\chi_3)$, we have
\begin{equation}
    \notag
     \sum_{n=0}^\infty b_3\left( \frac{mn-1}{12} \right)q^n\in S_{3m-3}(\Gamma_0(432),\chi_{12})_m.
\end{equation}

\end{proof}

\begin{proof}[Proof of Theorem \ref{cusp form2}]

For a fixed prime $m$, let
\begin{equation}
    \notag
    f(m;z):=\frac{\eta(5z)}{\eta(z)}\eta^a(5mz)\eta^b(mz),
\end{equation}
\noindent
where $m':=(m\ \mathrm{mod}\ 6)$ and $a:=5-m',\ b:=m'-1$. It is easy to show that $\modulo{f(m;z)}{\eta^{am+1}(5z)\eta^{bm-1}(z)}{m}$ and
$$\eta^{am+1}(5z)\eta^{bm-1}(z)\in S_{2m}(\Gamma_0(5),\chi_5),$$

\noindent
where $\chi_5(n)=\bracket{\frac{n}{5}}$. On the other hand,
\begin{equation}
    \notag
    f(m;z)=\sum_{n=0}^\infty b_5(n)q^{\frac{24n+m(5a+b)+4}{24}}\cdot\prod_{n=1}^\infty(1-q^{5mn})^a(1-q^{mn})^b.
\end{equation}
\noindent
Acting the $U(m)$ operator on $f(z)$ and since $\modulo{U(m)}{T(m)}{m}$, obtaining
\begin{equation}
    \label{after U/T}
    \modulo{\sum_{n=0}^\infty b_5(n)q^{\frac{24n+m(5a+b)+4}{24}}\ |\ U(m)}{\frac{\eta^{am+1}(5z)\eta^{bm-1}(z)\ |\ T(m)}{\prod_{n=1}^\infty(1-q^{5n})^a(1-q^{n})^b}}{m},
\end{equation}

\noindent
where $T(m)$ denotes usual Hecke operator acting on $S_{2m}(\Gamma_0(5),\chi_5)$. As for the LHS of (\ref{after U/T}), we have
\begin{equation}
    \notag
    \sum_{n=0}^\infty b_5(n)q^{\frac{24n+m(5a+b)+4}{24}}\ |\ U(m)=\sum_{\genfrac{}{}{0pt}{}{n=0}{m|6n+1}}^\infty b_5(n)q^{\frac{24n+m(5a+b)+4}{24m}}.
\end{equation}

\noindent
Using Theorem \ref{eta-quotient} and \ref{order of cusp}, one can verify that $\eta^4(5z)\eta^4(z)\in S_4(\Gamma_0(5))$ and have the order of $1$ at all cusps. Thus we can write $\eta^{am+1}(5z)\eta^{bm-1}(z)\ |\ T(m)=\eta^4(5z)\eta^4(z)g(m;z)$, where $g(m;z)\in M_{2m-4}(\Gamma_0(5),\chi_5)$. Hence
\begin{equation}
    \notag
    \modulo{\sum_{\genfrac{}{}{0pt}{}{n=0}{m|6n+1}}^\infty b_5(n)q^{\frac{6n+1}{6m}}}{\eta^{4-a}(5z)\eta^{4-b}(z)g(m;z)}{m}.
\end{equation}

\noindent
Replacing $q$ by $q^6$ shows that
\begin{equation}
    \notag
    \modulo{\sum_{\genfrac{}{}{0pt}{}{n=0}{m|6n+1}}^\infty b_5(n)q^{\frac{6n+1}{m}}}{\eta^{4-a}(30z)\eta^{4-b}(6z)g(m;6z)}{m}.
\end{equation}

\noindent
Since $b_5(n)$ vanishes for non-integer $n$, so
\begin{equation}
    \notag
    \modulo{\sum_{n=0}^\infty b_5\bracket{\frac{mn-1}{6}}q^{n}}{\eta^{4-a}(30z)\eta^{4-b}(6z)g(m;6z)}{m}.
\end{equation}

\noindent
Moreover, one can verify that $\eta^{4-a}(30z)\eta^{4-b}(6z)\in S_2(\Gamma_0(180))$. Together with $g(m;6z)\in M_{2m-4}(\Gamma_0(30),\chi_5)$, we have
\begin{equation}
    \notag
     \sum_{n=0}^\infty b_5\left( \frac{mn-1}{6} \right)q^n\in S_{2m-2}(\Gamma_0(180),\chi_{5})_m.
\end{equation}

\end{proof}

We need some important results due to Serre (c.f. \cite[(6.4)]{serre1974divisibilite}), which are the critical factors of the existence of Ramanujan-type congruences.

\begin{theorem}[J.-P. Serre]
\label{Serre's theorem}
The set of primes $l\equiv-1\ (\mathrm{mod}\ Nm)$ such that
$$f\ |\ T(l)\equiv0\ (\mathrm{mod}\ m)$$

\noindent
for each $f(z)\in S_k(\Gamma_0(N),\psi)_m$ has positive density, where $T(l)$ denotes the usual Hecke operator acting on $S_k(\Gamma_0(N),\psi)$.
\end{theorem}

Now Theorem \ref{infinitely many Ramanujan-type congruences} is an immediately corollary of the next two theorems.

\begin{theorem}
\label{main theorem}
Let $m\geq5$ be a prime. A positive density of the primes $l$ have the property that
\begin{equation}
    \notag
    b_3\left( \frac{mln-1}{12} \right)\equiv0\ (\mathrm{mod}\ m)
\end{equation}

\noindent
for each nonnegative integer $n$ coprime to $l$.
\end{theorem}

\begin{theorem}
\label{main theorem2}
Let $m\geq5$ be a prime. Then a positive density of primes $l$ have the property that
\begin{equation}
    \notag
    b_5\left( \frac{mln-1}{6} \right)\equiv0\ (\mathrm{mod}\ m)
\end{equation}

\noindent
satisfied for each integer $n$ coprime to $l$.
\end{theorem}

~

\begin{proof}[Proof of Theorem \ref{main theorem}]Let
\begin{equation}
    \notag
    F(m;z)=\sum_{n=0}^\infty b_3\left( \frac{mn-1}{12} \right)q^n,
\end{equation}

\noindent
then $F(m;z)\in S_{3m-3}(\Gamma_0(432),\chi_{12})_m$.

For a fix prime $m\geq5$, let $S(m)$ denote the of primes $l$ such that
$$f\ |\ T(l)\equiv0\ (\mathrm{mod}\ m)$$

\noindent
for each $f\in S_{3m-3}(\Gamma_0(432),\chi_{12})$. By Theorem \ref{Serre's theorem}, $S(m)$ contains a positive density of primes. So if $l\in S(m)$, we have
$$F(m;z)\ |\ T(l)\equiv0\ (\mathrm{mod}\ m).$$

\noindent
Then by the theory of Hecke operator we have
\begin{equation}
    \notag
    F(m;z)\ |\ T(l)=\sum_{n=0}^\infty\left( b_3\left( \frac{mln-1}{12} \right)+\left(\frac{3}{l}\right)l^{3m-4}b_3\left( \frac{mn-l}{12l} \right) \right)q^n\equiv0\ (\mathrm{mod}\ m).
\end{equation}

Since $b_3(n)$ vanishes when $n$ is not an integer, $b_3\left((mn-l)/12l\right)=0$ for each $n$ coprime to $l$ and $l\neq m$. Thus
$$b_3\left(\frac{mln-1}{12}\right)\equiv 0\ (\mathrm{mod}\ m)$$

\noindent
satisfied for each integer $n$ coprime to $l$ with $l\neq m$. Moreover, the set of such primes $l$ has a positive density of primes.

\end{proof}

\begin{proof}[Proof of Theorem \ref{main theorem2}]

Let
\begin{equation}
    \notag
    F(m;z)=\sum_{n=0}^\infty b_5\left( \frac{mn-1}{6} \right)q^n\in S_{2m-2}(\Gamma_0(180),\chi_5)_m.
\end{equation}

\noindent
By Theorem \ref{Serre's theorem}, the set of primes $l$ such that
$$ F(m;z)\ |\ T(l)\equiv0\ (\mathrm{mod}\ m)$$

\noindent
has positive density, where $T(l)$ denotes Hecke operator acting on $S_{2m-2}(\Gamma_0(180),\chi_5)$. Moreover, by the theory of Hecke operator, we have
\begin{equation}
    \notag
    \sum_{n=0}^\infty F(m;z)\ |\ T(l)=\sum_{n=0}^\infty \bracket{b_5\left( \frac{mln-1}{6} \right)+\bracket{\frac{l}{5}}l^{2m-3}b_5\left( \frac{mn-l}{6l} \right)}q^n.
\end{equation}

Since $b_5(n)$ vanishes for non-integer $n$, $b_5((mn-l)/6l)=0$ when $(n,l)=1$ and $l\neq m$. Thus we obtain
\begin{equation}
    \notag
    \modulo{b_5\bracket{\frac{mln-1}{6}}}{0}{m}
\end{equation}

\noindent
satisfied for each integer $n$ with $(n,l)=1$ and $l\neq m$. Moreover, the set of such primes $l$ has a positive density of primes.

\end{proof}

Since the number of selections of $l$ is infinite, choose $l>3$. Replacing $n$ by $12nl+ml+12$, then we have $b_3(ml^2n+ml+(m^2l^2-1)/12)\equiv0\ (\mathrm{mod}\ m)$ satisfied for each nonnegative integer $n$. Similar way can be applied to $b_5(n)$. Hence we obtain Theorem \ref{infinitely many Ramanujan-type congruences}. Moreover, since the choices of $l$ is infinite, together with the Chinese Remainder Theorem and previous results, we obtain

~

\begin{corollary}
If $m$ is a squarefree integer, then there are infinitely many Ramanujan-type congruences of $b_3(n)$ modulo $m$; if $k$ is a squarefree integer coprime to $3$, then there are infinitely many Ramanujan-type congruences of $b_5(n)$ modulo $k$.
\end{corollary}

\section{Distribution on nonzero residues}

Following Lovejoy \cite{lovejoy2001divisibility}, we need the following theorem due to Serre \cite{serre1974divisibilite}.

\begin{theorem}[J.-P. Serre]
\label{Serre's theorem v2}
The set of primes $l\equiv1\ (\mathrm{mod}\ Nm)$ such that
$$a(nl^r)\equiv(r+1)a(n)\ (\mathrm{mod}\ m)$$

\noindent
for each $f(z)=\sum_{n=0}^\infty a(n)q^n\in S_k(\Gamma_0(N),\psi)_m$ has positive density, where $r$ is a positive integer and $n$ is coprime to $l$.
\end{theorem}

~

Here we introduce a theorem of Sturm(Theorem 1 of \cite{sturm1987congruence}), which provide a useful criterion for deciding when modular forms with integer coefficients are congruent to zero modulo a prime via finite computation.

\begin{theorem}[J. Sturm]
\label{Sturm's theorem}
Suppose $f(z)=\sum_{n=0}^\infty a(n)q^n\in M_k(\Gamma_0(N),\chi)_m$ such that
$$a(n)\equiv0\ (\mathrm{mod}\ m)$$

\noindent
for all $n\leq \frac{kN}{12}\prod_{p|N}\left( 1+\frac1p \right)$. Then $a(n)\equiv0\ (\mathrm{mod}\ m)$ for all $n\in\mathbb{Z}$.
\end{theorem}

~

\begin{proof}[Proof of Theorem \ref{other residue classes}]

If there is one $k\in\mathbb{Z}$ such that
$$b_3\left( mk+\frac{m^2-1}{12} \right)\equiv e\not\equiv0\ (\mathrm{mod}\ m),$$

\noindent
let $s=12k+m$. Since $b_3(n)$ vanishes for negative $n$, we have $mk+\frac{m^2-1}{12}\geq0$. Hence $s=12k+m>0$ and
$$b_3\left(\frac{ms-1}{12}\right)=b_3\left( mk+\frac{m^2-1}{12} \right)\equiv e\ (\mathrm{mod}\ m).$$

\noindent
For a fix prime $m\geq5$, let $R(m)$ denote the set of primes $l$ such that
$$a(nl^r)\equiv(r+1)a(n)\ (\mathrm{mod}\ m)$$

\noindent
for each $f(z)=\sum_{n=0}^\infty a(n)q^n\in S_{3m-3}(\Gamma_0(432),\chi_{12})_m$, where $r$ is a positive integer and $n$ is coprime to $l$. By the proof of Theorem \ref{main theorem} we have $\sum_{n=0}^\infty b_3\left(\frac{mn-1}{12}\right)q^n\in S_{3m-3}(\Gamma_0(432),\chi_{12})_m$. Since $R(m)$ is infinite by Theorem \ref{Serre's theorem v2}, choose $l\in R(m)$ such that $l>s$, then
$$b_3\left(\frac{ml^rs-1}{12}\right)\equiv(r+1)b_3\left(\frac{ms-1}{12}\right)\equiv(r+1)e\ (\mathrm{mod}\ m).$$

\noindent
Now we fix $l$, choose $\rho\in R(m)$ such that $\rho>l$, then
\begin{equation}
    \label{rho}
    b_3\left(\frac{m\rho n-1}{12}\right)\equiv2b_3\left(\frac{mn-1}{12}\right)\ (\mathrm{mod}\ m)
\end{equation}

\noindent
satisfied for each $n$ coprime to $\rho$. For each $i=1,2,\cdots,m-1$, let $r_i\equiv i(2e)^{-1}-1\ (\mathrm{mod}\ m)$ and $r_i>0$. Let $n=l^{r_i}s$ in (\ref{rho}), we obtain
$$b_3\left(\frac{m\rho l^{r_i}s-1}{12}\right)\equiv2b_3\left(\frac{ml^{r_i}s-1}{12}\right)\equiv2(r_i+1)e\equiv i\ (\mathrm{mod}\ m).$$

Since the variables except $\rho$ are fixed, it suffice to prove that the estimate of the choices of $\rho\gg X/\log X$ and which is easily derived from Theorem \ref{Serre's theorem v2} and the Prime Number Theorem.

~

Moreover, by Sturm's Theorem, if $b_3\left( \frac{mn-1}{12} \right)\equiv0\ (\mathrm{mod}\ m)$ for each $n\leq216(m-1)$, then $b_3\left( \frac{mn-1}{12} \right)\equiv0\ (\mathrm{mod}\ m)$ for all $n\in\mathbb{Z}$. Since $b_3(n)$ vanishes if $n$ is not an integer, it suffice to compute those $n$ of the form $12j+m$ for $12j+m\leq216(m-1)$. This implies $j<18(m-1)$. In addition,
$$b_3\left( \frac{m(12j+m)-1}{12} \right)=b_3\left( mj+\frac{m^2-1}{12} \right).$$

\noindent
Thus if such $k$ exist, then $k<18(m-1)$.

\end{proof}

~

\begin{proof}[Proof of Theorem \ref{other residue classes2}]
The proof is similar to the proof above so we omit it.
\end{proof}

~

\section{Examples of Ramanujan-type congruences}

\label{examples}

\noindent
By Theorem \ref{Sturm's theorem} we find that
$$\sum_{n=0}^\infty b_3\left(\frac{mn-1}{12}\right)q^n\ |\ T(l)\equiv0\ (\mathrm{mod}\ m)$$

\noindent
for the pairs $(m,l)=(5,61)$, $(7,71)$, $(11,12553)$. An elementary computation yields that

~

\noindent
\begin{proposition}
\begin{equation}
    \notag
    b_3(18605n+127)\equiv0\ (\mathrm{mod}\ 5),
\end{equation}
\begin{equation}
    \notag
    b_3(35287n+207)\equiv0\ (\mathrm{mod}\ 7),
\end{equation}
\begin{equation}
    \notag
    b_3(1733355899n+126576)\equiv0\ (\mathrm{mod}\ 11).
\end{equation}
\end{proposition}

~

\noindent
Our method is not available to the case $m=3$, but one can prove that there are infinitely many Ramanujan-type congruences modulo $3$ via results of Lovejoy and Penniston \cite[Corollary 4]{lovejoy20013}.

~

\noindent
\begin{proposition}
If $m$ is a prime of the form $12k+1$, then
$$b_3\left(m^3n+\frac{m^2-1}{12}\right)\equiv0\ (\mathrm{mod}\ 3).$$
\end{proposition}

~

\noindent
For example, we obtain
$$b_3\left(2197n+14\right)\equiv0\ (\mathrm{mod}\ 3).$$

~

\noindent
As for $b_5(n)$, we compute that
$$\modulo{\sum_{n=0}^\infty b_5\bracket{\frac{mn-1}{6}}q^{n}\ |\ T(l)}{0}{m}$$

\noindent
satisfied for $(m,l)=(7,17)$, $(11,41)$, $(13,16519)$. An elementary computation yields that

~

\begin{proposition}
\begin{equation}
    \notag
    \modulo{b_5(2023n+99)}{0}{7},
\end{equation}
\begin{equation}
    \notag
    \modulo{b_5(18491n+75)}{0}{11},
\end{equation}
\begin{equation}
    \notag
    \modulo{b_5(3547405693n+35791)}{0}{13}.
\end{equation}
\end{proposition}

~

\noindent
Moreover, the congruence $\modulo{b_5(5n+4)}{0}{5}$ implies that
$$\modulo{\sum_{n=0}^\infty b_5\bracket{\frac{5n-1}{6}}q^{n}\ |\ T(l)}{0}{5}$$

\noindent
satisfied for each prime $l$.

~

\section{More on \textit{k}-regular partitions}

In this paper, we prove that for $b_k(n)(k=3,5)$ and each prime $m\geq5$ that there are infinitely many Ramanujan-type congruences modulo $m$. In fact, we conjecture that

~

\begin{conjecture}
For $b_k(n)(k=3,5)$ and each positive integer $m$ that there are infinitely many Ramanujan-type congruences modulo $m$.
\end{conjecture}

~

We also have the following conjecture analogous to Newman's Conjecture.

~

\begin{conjecture}
If $m$ is an integer, $k=3,5$, then for each residue class $r\ (\mathrm{mod}\ m)$ there are infinitely many integers $n$ for which $b_k(n)\equiv r\ (\mathrm{mod}\ m)$.
\end{conjecture}

Though Ramanujan-type congruences modulo primes $m\geq5$ exist, one may need enormous computation to find some. We encourage interested readers to find examples of congruences modulo other primes.

One can modify our proof to get some partial results of $b_{11}(n)$, the number of $11$-regular partitions of $n$. In fact, if $p$ is a prime for which $p>5$ and $\modulo{p}{5,7}{12}$, then $b_{11}(n)$ has infinitely many Ramanujan-type congruences modulo $p$. For example, we obtain $\modulo{b_{11}(43687n+230)}{0}{7}$. However, one can do better since from \cite{gordon1997divisibility} we have $\modulo{b_{11}(11n+6)}{0}{11}$. We intent to take up these in a future paper.

~

\section*{Acknowledgement}

The ideas came to us after seeing the paper of Ono \cite{ono2000distribution} and Lovejoy \cite{lovejoy2001divisibility}.

\begin{comment}

\section*{Declarations}

Conflict of interest statement: We declare that we do not have any commercial or associative interest that represents a conflict of interest in connection with the work submitted.

Data availability statement: Not applicable.

\end{comment}

~

% \printbibliography

\end{document}